\documentclass[12pt,reqno]{amsart}
\usepackage{
amsmath,
amssymb,
enumerate,
fullpage,
mathrsfs,
mathtools,
stmaryrd,
thmtools,
verbatim}

\usepackage{hypbmsec}
\usepackage[dvipsnames]{xcolor}
\usepackage[breaklinks=true,colorlinks=true,
linkcolor=black!25!RoyalBlue,citecolor=black!25!RoyalBlue,
filecolor=RoyalBlue,urlcolor=teal,pagebackref]{hyperref}

\hypersetup{linktocpage}
\usepackage[capitalise]{cleveref}

\usepackage{tikz}
\usepackage{tikz-cd}
\usetikzlibrary{shapes.geometric}
\usetikzlibrary{shapes.misc}
\usetikzlibrary{positioning}

\newtheorem{theorem}{Theorem}[subsection]
\newtheorem{corollary}[theorem]{Corollary}
\newtheorem{proposition}[theorem]{Proposition}
\newtheorem{lemma}[theorem]{Lemma}
\newtheorem{lettertheorem}{Theorem}

\newtheorem{definition}[theorem]{Definition}
\theoremstyle{definition}

\theoremstyle{remark}
\declaretheorem[name=Remark,sibling=theorem,qed={\lower-0.3ex\hbox{$\diamond$}}]{remark}

\declaretheorem[name=Notation,unnumbered,qed={\lower-0.3ex\hbox{$\diamond$}}]{notation}

\numberwithin{equation}{section}

\renewcommand{\AA}{\mathbb{A}}
\newcommand{\Af}{\AA_{\mathrm{f}}}
\newcommand{\CC}{\mathbb{C}}
\newcommand{\FF}{\mathbb{F}}
\newcommand{\GG}{\mathbb{G}}
\newcommand{\QQ}{\mathbb{Q}}
\newcommand{\Qp}{\QQ_p}
\newcommand{\Qbar}{\overline{\QQ}}

\newcommand{\TT}{\mathbb{T}}
\newcommand{\ZZ}{\mathbb{Z}}
\newcommand{\Zp}{\ZZ_p}

\newcommand{\fm}{\mathfrak{m}}
\newcommand{\fp}{\mathfrak{p}}
\newcommand{\fq}{\mathfrak{q}}

\newcommand{\fN}{\mathfrak{N}}
\newcommand{\fP}{\mathfrak{P}}
\newcommand{\fS}{\mathfrak{S}}

\newenvironment{smatrix}{\left( \begin{smallmatrix}} {\end{smallmatrix} \right)}
\newcommand{\tbt}[4]{\begin{pmatrix}#1 & #2 \\ #3 & #4\end{pmatrix}}
\newcommand{\stbt}[4]{\begin{smatrix}#1 & #2 \\ #3 & #4\end{smatrix}}

\DeclareMathOperator{\AF}{AF}
\DeclareMathOperator{\GL}{GL}

\DeclareMathOperator{\Gal}{Gal}
\DeclareMathOperator{\pr}{pr}

\DeclareMathOperator{\Hom}{Hom}
\DeclareMathOperator{\Res}{Res}
\DeclareMathOperator{\Frob}{Frob}

\DeclareMathOperator{\Iw}{Iw}

\DeclareMathOperator{\Nm}{Nm}
\DeclareMathOperator{\Spec}{Spec}
\DeclareMathOperator{\Sym}{Sym}

\DeclareMathOperator{\TSym}{TSym}

\DeclareMathOperator{\mom}{mom}
\usepackage{orcidlink}
\DeclareMathOperator{\myPr}{Pr}\renewcommand{\Pr}{\myPr}

\DeclareMathOperator{\As}{As}

\newcommand{\cA}{\mathcal{A}}

\newcommand{\cE}{\mathcal{E}}
\newcommand{\cF}{\mathcal{F}}

\newcommand{\cH}{\mathcal{H}}

\newcommand{\cO}{\mathcal{O}}
\newcommand{\cR}{\mathcal{R}}
\newcommand{\cS}{\mathcal{S}}
\newcommand{\cT}{\mathcal{T}}
\newcommand{\cU}{\mathcal{U}}

\newcommand{\cW}{\mathcal{W}}

\newcommand{\uPi}{\underline{\Pi}}

\newcommand{\ch}{\mathrm{ch}}

\newcommand{\et}{\text{\textup{\'et}}}

\newcommand{\into}{\hookrightarrow}

\newcommand{\sH}{\mathscr{H}}

\newcommand{\sph}{\mathrm{sph}}
\newcommand{\ord}{\mathrm{ord}}
\newcommand{\rig}{\mathrm{rig}}
\newcommand{\ad}{\mathrm{ad}}

\numberwithin{equation}{section}
\renewcommand{\le}{\leqslant}

\renewcommand{\ge}{\geqslant}
\renewcommand{\geq}{\geqslant}

\author[David Loeffler]{David Loeffler \orcidlink{0000-0001-9069-1877}} 
\address[David Loeffler]{Faculty of Mathematics and Computer Science, UniDistance Switzerland, Schinerstrasse 18, CH-3900 Brig, Switzerland}
\email{david.loeffler@unidistance.ch}

\author[Arshay Sheth]{Arshay Sheth \orcidlink{0000-0002-8271-4036}}
\address[Arshay Sheth]{Mathematics Institute, Zeeman Building, University of Warwick, Coventry, CV4 7AL, UK}
\email{arshay.sheth@warwick.ac.uk}

\thanks{Both authors gratefully acknowledge financial support from the European Research Council under the European Union’s Horizon 2020 research and innovation programme (Grant agreement No. 101001051 — Shimura varieties and the Birch--Swinnerton-Dyer conjecture).}
\date{}

\title{The Asai--Flach Euler system in $p$-adic families}

\begin{document}

\setlength\marginparwidth{10pt}
\newcommand{\DLnote}[1]{{\color{red}\marginpar{\textcolor{red}{$\spadesuit$}} #1 -- DL}}
\newcommand{\AS}[1]
{\textcolor{purple}{\textsf{AS: [#1]}}}

\begin{abstract}
 We show that the Euler system for the Asai representation corresponding to a Hilbert modular eigenform over a real quadratic field, constructed by Lei, Loeffler and Zerbes (2018), can be interpolated $p$-adically as the Hilbert modular form varies in a Hida family. This work is used as an important input in recent work of Grossi, Loeffler and Zerbes (2025) on the proof of the Bloch--Kato conjecture in analytic rank zero for the Asai representation. 
\end{abstract}

 \maketitle

\setcounter{tocdepth}{1} \tableofcontents

\section{Introduction}

  This article is a sequel to the second author's paper \cite{She25} and is motivated by the theme of using $p$-adic variation of global cohomology classes to make progress on conjectures about special values of $L$-functions. This theme has its origins in the work of Bertolini--Darmon--Rotger \cite{BDR15a, BDR15b}, where they studied Beilinson--Flach elements, which are canonical global cohomology classes in the cohomology of the tensor product of $p$-adic Galois representations of two weight two modular forms. They showed that these classes vary $p$-adically in Hida families, and used this in a crucial way to prove the Birch and Swinnerton--Dyer conjecture in analytic rank zero for the $\rho$-isotypic component of the Mordell--Weil group of an elliptic curve, where $\rho$ is an odd irreducible Artin representation. Subsequent work by Lei--Loeffler--Zerbes \cite{LLZ14} constructed an Euler system, a compatible collection of cohomology classes over cyclotomic fields, of which the Beilinson--Flach elements are the bottom layer. This construction was extended to pairs of higher-weight modular forms by Kings--Loeffler--Zerbes in \cite{KLZ20}. 
 
  In both cases, the authors showed that the entire Euler system also varies $p$-adically in Hida families, and used this to derive further arithmetic applications. For instance, Kings--Loeffler--Zerbes proved the Bloch--Kato conjecture in analytic rank zero for the Galois representation associated to the tensor product of two modular forms, and as a corollary proved the finiteness of the Selmer group of an elliptic curve twisted by a two-dimensional odd irreducible Artin representation when the associated $L$-value does not vanish. 
 
  In this paper, we consider the Asai--Flach Euler system, introduced in \cite{LLZ18}, which is an Euler system for the Asai representation corresponding to a Hilbert modular eigenform over a real quadratic field. We show that this Euler system interpolates in Hida families. In a similar spirit to the works described above, our results have found applications in recent work of Grossi--Loeffler--Zerbes \cite{GLZ25} where they prove the Bloch--Kato conjecture in analytic rank zero for the Asai representation. Although our results are similar in outline to those described above for the Rankin--Selberg case, the arguments are significantly more complicated owing to the role of the global units of the field. We describe our main results in more detail below. 

 \subsection{Overview of results} 
  
  Our first main result is an extension of a theorem from \cite{She25}, interpolating the cohomology of Hilbert modular varieties in families. Let $F$ be a totally real field unramified above $p$; and let $\cO$ be the ring of integers of some large enough extension of $\Qp$ (containing all embeddings of $F$). For any sufficiently small tame level $K^{(p)}$, we show in \S \ref{controlthm} that there exists a graded module $e'_{\ord} H^*_{\Iw}(Y_G(K_\infty)_{\Qbar}, \cO)$, with actions of Hecke operators, the Galois group $\Gal(\Qbar / \QQ)$, and an Iwasawa algebra $\Lambda$ in $[F : \QQ] + 1 + \delta_F$ variables (where $\delta_F$ is the Leopoldt defect for $F$). 
  
  \begin{lettertheorem}[Theorem \ref{thm:classicity}]
   There is a Tor spectral sequence relating the module  \newline $e'_{\ord} H^*_{\Iw}(Y_G(K_\infty)_{\Qbar}, \cO)$  to the $U_p'$-ordinary part of the \'etale cohomology of the Shimura variety of Iwahori level at $p$ with coefficients in any algebraic representation of $\operatorname{Res}_{F/\QQ}\GL_2$ whose central character factors through the norm map.
  \end{lettertheorem}
  
  This extends an earlier construction in \cite{She25} in which the coefficient sheaf was assumed to have trivial central character (and the Iwasawa algebra had only $[F : \QQ]$ variables). Closely related results are already in the literature due to Hida \cite{Hid86,Hid88}, Dimitrov \cite{Dim13}, and Mauger \cite{Mau04}; but our formulation is somewhat different (inspired by the work of Ohta \cite{ohta99} for the modular curve) and is more convenient for relations to Euler systems.
  
  For our remaining results we suppose $[F : \QQ] = 2$.
  
  \begin{lettertheorem}[Theorem \ref{b2}] \label{b1}
   We can find a cohomology class $$z_\infty \in H^1_{\Iw}(\QQ(\mu_{p^\infty}), e'_{\ord} H^2_{\Iw}(Y_G(K_\infty)_{\Qbar}, \cO)(2)),$$ whose specializations at algebraic weights $(\lambda, h)$ for $\Res_{F/\QQ}(\GL_2) \times \GL_1$ satisfying suitable inequalities are the Asai--Flach elements constructed in \cite{LLZ18}.
  \end{lettertheorem}
  
  Theorem \ref{b1} is fairly straightforward to prove once the cohomology groups involved have been constructed; but it is not so easy to apply in practice, firstly since the module structure of $H^*_{\Iw}(Y_G(K_\infty)_{\Qbar}, \cO)$ may be complicated, and secondly because the construction only sees Asai--Flach classes with nontrivial levels at $p$. Hence we also prove a much more precise theorem after projecting to the eigenspace corresponding to an ordinary family $\uPi$ over an affinoid disc $\mathcal{C}$ in weight space.
  
  \begin{lettertheorem}[Theorem \ref{c2}]
   There exists a family of Galois representations $M(\uPi)^*$ which is free of rank 4 over $\mathcal{C}$, whose specializations at algebraic weights $\lambda \in \mathcal{C}$ are canonically isomorphic to the Asai Galois representations associated to the specializations $\Pi[\lambda]$ of $\uPi$. Moreover, there is a class
   \[ {}_c \mathcal{AF}(\uPi) \in H^1_{\Iw}(\QQ(\mu_{p^\infty}), M(\uPi)^*)\]
   whose specializations at appropriate characters $(\lambda, h)$ are explicit scalar multiples of the (prime-to-$p$ level) Asai--Flach classes associated to the specializations $\uPi[\lambda]$.
  \end{lettertheorem}
  
%
%
%

\section{Background on Hilbert modular varieties}

 \subsection{Notation}
 
  We let $F$ be a totally real field of degree $d$ with ring of integers $\cO_F$ and discriminant $\Delta_F$; and we let $p$ be an odd prime unramified in $F$. We fix a numbering $\{\sigma_1, \dots, \sigma_d \}$ of the embeddings $F \hookrightarrow \mathbb{R}$. 
  
  We fix an isomorphism $\CC \cong \Qbar_p$, so that $\sigma_i$ can be regarded as embeddings $F \into \Qbar_p$; and let $E \subset \Qbar_p$ be a finite extension of $\QQ_p$ containing the images of the $\sigma_i$. We let $\cO$ be the integer ring of $E$.
  
  Let $\cH$ denote the upper half plane and let $\cH_F$ denote the set of elements of $F \otimes \CC$ of totally positive imaginary part; note that $\cH_F$ can be identified with the product of $d$ copies of $\cH$. We let $G$ be the algebraic group $\Res^F_{\QQ} \GL_2$ over $\QQ$.%

  \subsection{Shimura varieties and Hecke actions}

  For $K$ an open compact subgroup of $G(\Af)$, we have associated Shimura varieties
  \begin{align*} 
   Y_G(K)(\CC) &= G(\QQ)^+ \backslash [ G(\Af) \times \cH_F]/K.
  \end{align*}
  
  We shall always assume $K$ is \emph{sufficiently small} in the sense\footnote{Note that the corresponding definition in \cite{LLZ18} has a typo ($u \in \cO_F^{\times +}$); the formula in \cite{She25} is the correct one.} of \cite[Definition 2.2]{She25}, so that $Y_G(K)$ is a smooth variety over $\QQ$.

  We set $\TT_K(G) \coloneqq \ZZ[K\backslash G(\Af)/K]$ to be the Hecke algebra, acting on the right on $Y_G(K)$ via correspondences as in \cite[\S 2.5]{She25}.
  We obtain a left $\TT_K(G)$-module structure on $H^i_{\et}(Y_G(K)_{\Qbar}, \mathbb{F}_p)$ and $H^i_{\et}(Y_G(K)_{\Qbar}, \Zp)$ by contravariant functoriality. As in \emph{op.cit.}~we denote by $\TT \subset \TT_K(G)$ the spherical Hecke algebra, which is the commutative subalgebra generated by operators $\{\cT_v, \cS_v, \cS_v^{-1}\}$ for each prime $v$ at which $K$ is unramified.

 \subsection{Coefficient sheaves}
  \label{repns}
 
  For each dominant weight $\lambda$ of $G$, there is a unique (up to isomorphism) irreducible representation $V_\lambda$ of $G$ over $\Qbar$ of highest weight $\lambda$; and these representations are definable over $E$.
 
  We will only consider representations whose central character factors through the norm map; concretely, these are given by $\bigotimes_{i=1\dots d} (\Sym^{k_i} \otimes \det^{t_i})$ for integers $(k_1, \dots, k_d, t_1, \dots, t_d)$, with $k_i \ge 0$, satisfying the constraint $k_1 + 2t_1 = \dots = k_d + 2 t_d = w$ for some $w$. We call such weights \emph{pure}. For any sufficiently small level $K$, we have a functor (see \cite[\S 4.4]{LLZ18})
  \[
   \{\textrm{pure algebraic representations of $G$}\} \rightarrow \{\textrm{lisse \'etale $E$-sheaves on }Y_G(K) \}, 
  \]
  that maps the representation $V_{\lambda}$ to the \'etale sheaf $\sH^{ (\lambda)}$. We write $\sH^{[\lambda]}$ for the dual of $\sH^{(\lambda)}$. The cohomology of these sheaves has a natural action of $\TT_K(G)$; and for primes $\fq$ trivial in the ray class group modulo $Z_G \cap K$, $\cS_\fq$ acts as $N(\fq)^w$ on $\sH^{(\lambda)}$, and $N(\fq)^{-w}$ on $\sH^{[\lambda]}$.
  
  \begin{remark}
   Note that if $\lambda = (1, \dots, 1, 0, \dots, 0)$ for some $k$, then the pullback of $\sH^{(\lambda)}$ to the PEL Shimura variety for the subgroup $G^* \subset G$ (see \cref{def:Gstar} below) is a direct summand of the $d$-th \'etale cohomology of the universal abelian variety, while $\sH^{[\lambda]}$ corresponds to the $d$-th \'etale homology.
  \end{remark}
  \subsubsection*{Integral structures}
  
   We can define lisse \'etale sheaves of $\cO$-modules $\sH^{[\lambda]}_{\cO}$ (with $\sH^{[\lambda]} = \sH^{[\lambda]}_{\cO} \otimes_{\cO} E$) using any choice of \emph{admissible lattice} in $V_\lambda$, in the sense of \cite[Definition 2.3.1]{LRZ24}. We shall adopt the convention of using the \emph{minimal} admissible lattice, which corresponds to the symmetric tensors $\TSym^k$ used in \cite{LLZ18} (see Definition 6.1.2 of \emph{op.cit.}).
   
   Hence, for any $i$ and $K$, the cohomology group $H^i_{\et}\left(Y_G(K)_{\Qbar}, \sH^{[\lambda]}_{\cO}\right)$ is a finitely-generated $\cO$-module, and its image modulo torsion is an $\cO$-lattice in $H^i_{\et}\left(Y_G(K)_{\Qbar}, \sH^{[\lambda]}\right)$. These $\cO$-modules have natural actions of the prime-to-$p$ Hecke algebra $\ZZ[K^{(p)} \backslash G(\Af^{(p)}) / K^{(p)}]$; but they are not necessarily preserved by Hecke operators at $p$.

 \subsection{Normalized Hecke operators}
 
  We now suppose $K = K^{(p)} K_p$ for some $K_p \subset G(\Zp)$ having an Iwahori decomposition. For these levels we have the following Hecke operators: 
   
  \begin{definition}
   For $n \ge 0$ and any $\fp \mid p$, we write $\cU_\fp \in \ZZ[K_{n, p} \backslash G(\QQ_p) / K_{n, p}]$ for the double coset of $\tbt{\varpi_{\fp}}{}{}{1}$, where $\varpi_{\fp} \in (F \otimes \QQ_p)^\times$ maps to $p$ in $F_{\fp}$, and to $1$ in $F_{\fq}$ for primes $\fq \mid p$ other than $\fp$ (if any). Similarly, we write $\cU'_{\fp}$ for the double coset of $\tbt{\varpi_{\fp}^{-1}}{}{}{1}$. We define $\cU_p = \prod_{\fp \mid p} \cU_{\fp}$, and similarly $\cU_p'$.
  \end{definition}
  
  For non-trivial coefficient sheaves, the eigenvalues of these operators are not necessarily integral, so we shall renormalize the Hecke action.
   
  \begin{definition} We define the following Hecke operators on the cohomology of $\sH^{(\lambda)}$, for a pure weight $\lambda = (\underline{k}, \underline{t})$: \
   \begin{itemize}
    \item We let $U_p = p^h \cU_p$, where $h = -\sum_i t_i = \tfrac{1}{2}(\sum_i k_i - d w)$.
    \item For $\fp \mid p$, we let $I_{\fp}$ be the set of indices such that $\sigma_i : F \into \Qbar_p$ factors through $F_{\fp}$, and set $U_{\fp} = p^{h_{\fp}} \cU_{\fp}$, where $h_{\fp} = -\sum_{i \in I_{\fp}}t_i$.
   \end{itemize}
   These operators are well-defined on the cohomology of $\sH^{(\lambda)}_{\cO}$ (see \cite[\S 2.5]{LRZ24}). Similarly we define $U_p' = p^h \cU_p'$ and $U_{\fp}' = p^{h_{\fp}} \cU_\fp'$ acting on the cohomology of $\sH^{[\lambda]}_{\cO}$.
  \end{definition}
  
  Note that this normalization depends on our choice of embedding of the coefficient field into $\Qbar_p$. We define $e'_{\ord}$ for the ordinary projector $\lim_{n \to \infty} (U_p')^{n!}$ associated to $U_p'$. 

 \subsection{Asai Galois representations}

  Let $\lambda = (\underline{k}, \underline{t})$ be a pure weight, and let $\cF : G(\Af) \times \cH_F \to \CC$ be a Hilbert modular eigenform over $F$ of level $\fN \subseteq \cO_F$ and weight $(\underline{k} + (2, \dots, 2), \underline{t})$ (in the conventions of \cite[Section 4.1]{LLZ18}). Then the translates of $\cF$ generate an automorphic representation $\Pi$ of $G(\AA_F)$ which is cohomological of weight $\lambda$; and the central character of $\Pi$ maps a uniformizer at a prime $\fq$ to $N(\fq)^w \omega_{\Pi}(\fq)$, where $w \coloneqq k_1 + 2t_1 = \dots = k_d + 2t_d$ and $\omega_{\Pi}$ is a finite-order Hecke character modulo $\fN$ (the nebentype of $\cF$).
 
  \begin{remark}
   The $w$ of \cite{LLZ18} is $w + 2$ in the present notation, so this statement is consistent with the last paragraph of \S 4.1 of \emph{op.cit.}. (The ``$\Nm_{F/\QQ}(u)^{1-w}$'' in the paragraph following Notation 4.1.1 of \emph{op.cit.} is a typo, the exponent should be $2 - w$.)
  \end{remark}

  The representation $\Pi$ determines a ring homomorphism $\lambda_{\Pi} : \TT \to \CC$, where $\TT$ is the spherical Hecke algebra away from $\fN$, sending each Hecke operator to its eigenvalue acting on $\cF$; and the image of $\lambda_{\Pi}$ is contained in a finite extension of $\QQ$ (depending on $\Pi$). Via our isomorphism $\iota: \CC \cong \Qbar_p$ we regard $\lambda_{\Pi}$ as taking values in $\Qbar_p$. 
  
  \begin{theorem}[Blasius--Rogawski--Taylor]
   There exists a unique (up to isomorphism) irreducible 2-dimensional Galois representation
   \[ 
    \rho_{\Pi, \iota}: \Gal(\overline{F}/F) \rightarrow \GL_2(\Qbar_p) 
   \]
   such that for all primes $\fq \nmid \fN p$, the representation $\rho_{\Pi, \iota}$ is unramified at $\fq$ and we have 
   \[ \det(1-X\rho_{\Pi, \iota}(\Frob_\fq^{-1})) = 1 - \lambda_{\Pi}(\cT_{\fq})\cdot X + N(\fq) \lambda_{\Pi}(\cS_\fq) \cdot X^2,\]
   where $\Frob_\fq$ is the arithmetic Frobenius.
  \end{theorem}
  
  We recall that if $H \subseteq G$ are groups such that $[G:H] < \infty$, and $\rho$ is a representation of $H$ on some vector space $V$, there exists a representation $\left(\bigotimes\textrm{-Ind}\right)_H^{G}(\rho)$ of $G$ whose underlying vector space is $V^{\otimes [G : H]}$ with an appropriate action of $G$. If $[G : H] = 2$, this can be given explicitly by fixing a choice of $\sigma \in G - H$ and defining
  \[
   h \cdot (v \otimes w)= (h \cdot v) \otimes (\sigma^{-1}h\sigma \cdot w), \hspace{2mm} \sigma \cdot (v \otimes w) = (\sigma^2 \cdot w) \otimes v. 
  \]

  \begin{definition}
   We define the $2^d$-dimensional Asai Galois representation 
   \[
    \rho_{\Pi, \iota}^{\As}: \Gal(\Qbar/\QQ) \rightarrow \GL_4(\Qbar)
   \]
   by $\rho_{\Pi, \iota}^{\As} \coloneqq \left(\bigotimes\text{-}\mathrm{Ind}\right)_{\Gal(\overline{F}/ F)}^{\Gal(\Qbar/\QQ)} (\rho_{\Pi, \iota})$. 
  \end{definition}

  This has a canonical realisation in \'etale cohomology. Let us suppose our coefficient field $E$ contains the image of $\lambda_{\Pi}$. We write $Y_1(\fN)$ for the Shimura variety of level $\{ \stbt{*}{*}{0}{1} \bmod \fN\}$.

  \begin{theorem}[{Brylinski--Deligne, Nekov\'a\v{r}; see \cite[Theorem 2.2]{LLZ18}}]
   The space
   \[ 
    M_E(\Pi)\coloneqq H^d_{\et, c}\left(Y_1(\fN)_{\Qbar}, \sH^{(\lambda)}\right)[I_{\Pi}],
   \]
   where $I_{\Pi}$ is the kernel of the map $\lambda_{\Pi} \otimes 1 : \TT \otimes_{\ZZ} E \to E$, is a $2^d$-dimensional $E$-linear representation of $\Gal(\Qbar/\QQ)$ whose base-extension to $\Qbar$ is a canonical representative of the isomorphism class $\rho_{\Pi, \iota  }^{\As}$.
  \end{theorem}
  
  Via Poincar\'e duality, we have a canonical isomorphism
  \[ M_{E}(\Pi)^* \cong H^d_{\et}\left(Y_1(\fN)_{\Qbar}, \sH^{[\lambda]}(d)\right) \otimes_{\TT} \TT / \nu(I_{\Pi}), \]
  where $\nu$ is the anti-automorphism of the Hecke algebra given by $g \mapsto g^{-1}$.

  \begin{remark} 
   Note that the representation denoted $M_{L_v}(\cF)$ in \cite{LLZ18}, for $v \mid p$ the prime determined by $\iota$, is $M_{L_v}(\Pi)(t_1 + t_2)$ in our present notation.
  \end{remark}
  
 \subsection{P-stabilization maps}

  For representations $\Pi$ of prime-to-$p$ conductor, we shall need to compare the Galois representation $M_E(\Pi)$ with a related space defined using Iwahori level at $p$. Compare e.g.~\cite[\S 5.7]{KLZ17}. For $\fp \mid p$, write $\Iw_\fp$ for the upper-triangular Iwahori subgroup of $\GL_2(\cO_{F, \fp})$.
  
  \begin{definition}
   For each $\fp \mid p$, consider the two elements of $\ZZ[ \Iw_{\fp} \backslash \GL_2(F_{\fp}) / \GL_2(\cO_{F,\fp})]$ given by
   \[ \pr_{\fp, 1} \coloneqq [\mathrm{id}], \qquad \pr_{\fp, 2} \coloneqq [\stbt{1}{0}{0}{p}].\]
  \end{definition}
  
  We compute that
  \begin{align*}
   \cU_\fp \cdot \pr_{\fp, 1} &= \pr_{\fp, 1}\mathop{\cdot} \cT_\fp - \pr_{\fp, 2},&
   \cU_\fp \cdot \pr_{\fp, 2} &= N(\fp) \pr_{\fp_1} \mathop{\cdot} \cS_{\fp}.
  \end{align*}
  So if $\Pi_\fp$ is an irreducible unramified $\GL_2(F_{\fp})$-representation (defined over $E$), and we choose a root $\alpha_{\fp} \in E$ of $X^2 - \cT_{\fp} X + N(\fp) \cS_{\fp}$ acting on $\Pi_{\fp}^{\GL_2(\cO_{F,\fp})}$, then the Hecke operator
  \[ \Pr_{\fp, \alpha_{\fp}} \coloneqq \pr_{\fp, 1} - \tfrac{1}{\alpha_{\fp}} \pr_{\fp, 2} \in E[ \Iw_{\fp} \backslash \GL_2(F_{\fp}) / \GL_2(\cO_{F,\fp})]\]
  maps $\Pi_{\fp}^{\GL_2(\cO_{F,\fp})}$ isomorphically to the 1-dimensional subspace of $\Pi_{\fp}^{\Iw_{\fp}}$ on which $\cU_{\fp}$ acts as $\alpha_{\fp}$.
  
  For an unramified representation $\Pi_p = \prod_{\fp \mid p}$ of $G(\QQ_p)$, we choose a collection $\alpha = (\alpha_{\fp})_{\fp \mid p}$, which we term a \emph{$p$-refinement} of $\Pi_p$; and we let $\Pr_{p, \alpha} = \prod_{\fp} \Pr_{\fp, \alpha_{\fp}} \in E[\Iw_p \backslash G(\QQ_p) / G(\Zp)]$. We write $\Pr_{p, \alpha}^*$ for the map it induces on cohomology, and $\Pr_{p, \alpha, *}$ for its transpose (i.e.~the map induced by $\nu(\Pr_{p, \alpha})$).
  
  \begin{proposition}
   If $\Pi$ is a cohomological cuspidal automorphic representation of $G$ of level $\fN$, -to-$p$ level $\fN$, and $\alpha$ is a $p$-refinement of $\Pi$, then the map $\Pr_{p, \alpha}^*$ defines an isomorphism $M_E(\Pi) \xrightarrow{\ \cong\ } M_E(\Pi; \alpha)$, where 
   \[ M_E(\Pi; \alpha) \subseteq H^d_{\et}\left(Y_G(U_1(\fN)^{(p)} \Iw_p), \sH^{(\lambda)}\right)\]
   is the maximal subspace where the prime-to-$\fN p$ Hecke algebra $\TT^{(p)}$ acts via $\lambda_{\Pi}$, and $\cU_{\fp}$ acts as $\alpha_{\fp}$ for all $\fp \mid p$. 
   
   Similarly (in fact dually), $\Pr_{p, \alpha, *}$ gives an isomorphism
   $M_E(\Pi; \alpha)^* \xrightarrow{\ \cong\ }M_E(\Pi)^*$, where $M_E(\Pi; \alpha)^*$ is the maximal quotient of $H^d_{\et}\left(Y_G(U_1(\fN)^{(p)} \Iw_p), \sH^{[\lambda]}(2)\right)$ on which $\TT^{(p)}$ acts via $\lambda_{\Pi^\vee} = \lambda_{\Pi} \circ \nu$ and $\cU_{\fp}'$ acts as $\alpha_{\fp}\, \forall \fp \mid p$.\qed
  \end{proposition}
  
  (Note that if the $\alpha_{\fp}$ are simple roots of the Hecke polynomials, the quotient $M_E(\Pi; \alpha)^*$ of $H^d_{\et}\left(Y_G(U_1(\fN)^{(p)} \Iw_p), \sH^{[\lambda]}(2)\right)$ lifts to a direct summand. This is conjectured to always hold; but it is automatic in the ordinary case, since the roots of the Hecke polynomials have distinct valuations.)
  
\section{Control theorems} \label{controlthm}
 
 We now recall and extend some results from the paper \cite{She25} of the first author. Our goal is to show that not only can cuspidal Hilbert modular forms (and their Galois representations) be deformed in $p$-adic families (which recovers well-known results of Hida), but also that the Asai Galois representations associated to such families have \emph{canonical} realisations in the cohomology of a suitable tower of Shimura varieties for $G$. 
 
 \subsection{Ordinary cohomology modules} 
 
  Let $K^{(p)} \subset G\big(\Af^{(p)}\big)$ be a compact open such that $K^{(p)} G(\Zp)$ (and hence any subgroup thereof) is sufficiently small.
  
  \begin{definition}
   Let $\cE_{K^{(p)}} = \{u \in \cO_F^\times : \stbt{u}{}{}{u} \in K^{(p)}\}$; and for each $n \in \mathbb{N}_{\geq 1}$, we let 
   \[ 
    K_{n, p} \coloneqq \left\{ g \in \GL_2(\cO_F \otimes \Zp) : 
    g \equiv \tbt{u}{*}{0}{u} \bmod p^n 
    \text{ for some } u \in \cE_{K^{(p)}} \right\}.
   \]
   When $n=0$, we define $K_{0, p} = \Iw_p$ to be the Iwahori subgroup of $G(\Zp)$, \textit{i.e.}, the subgroup of matrices which are upper triangular mod $p$. We let $K_n \coloneqq K^{(p)} K_{n, p}$.
  \end{definition}
 
  \begin{remark}
   This definition is motivated by the fact that $G(\Zp)$ does not act freely on $\varinjlim_{K_p}Y_G(K^{(p)}K_p)$; the closure $\overline{\cE}_{K^{(p)}}$ acts trivially (and the ``sufficiently small'' condition shows that the action of $G(\Zp) / \overline{\cE}_{K^{(p)}}$ is free).
  \end{remark}
  
  \begin{definition}
   Set $H^i_{\Iw}(Y_G(K_\infty)_{\Qbar}, \cO) \coloneqq \varprojlim_{n \ge 0} H^i_{\et}(Y_G(K_\infty)_{\Qbar}, \cO)$, where the limit is with respect to the pushforward maps.
  \end{definition}
  
  These groups are equipped with actions by the following objects (all commuting with one another):
  \begin{itemize}
   \item the prime-to-$p$ Hecke algebra $\TT_{K^{(p)}}$;
   \item the operators $U'_{\fp}$ for $\fp \mid p$; 
   \item the Galois group $\Gal(\Qbar / \QQ)$ (with inertia groups outside $pN \Delta_F$ acting trivially);
   \item the group $\fS \coloneqq T(\Zp) / \overline{\cE}_{K^{(p)}}$, where $T$ is the diagonal torus.
  \end{itemize}
  To simplify some formulae later, we shall compose the natural pullback action of $\fS$ with conjugation by the long Weyl element $w_0 = \stbt{0}{-1}{1}{0}$. Note that $\fS$ is an abelian $p$-adic Lie group of dimension $d + 1 + \delta_F$, where $\delta_F \ge 0$ is the defect in Leopoldt's conjecture for $F$. 
  
  \begin{remark}
   Using Poincar\'e duality, one can check that $e'_{\ord} H^*_{\Iw}(Y_G(K_\infty)_{\Qbar}, \cO)$ coincides with the module denoted $\cH_{11}( 0)$ in \cite[\S 3.2]{Dim13} (which is defined as the Pontryagin dual of a direct limit with $E / \cO$ coefficients). 
  \end{remark}
  
  We write $\fS_n$ for the image of $T(\Zp) \cap K_{n, p}$ in $\fS$, so the $\fS_n$ are open subgroups with $\bigcap_n \fS_n = \{1\}$ (and $\fS_0 = \fS$). We set $\Lambda = \cO[[\fS]]$, and more generally $\Lambda_n = \cO[[\fS_n]]$.
 
  \begin{theorem}
   The groups $e'_{\ord} H^i_{\Iw}\left(Y_G(K_\infty)_{\Qbar}, \cO\right)$ are finitely-generated modules over the Iwasawa algebra $\Lambda = \cO[[ \fS ]]$.
  \end{theorem}
  
  \begin{proof}
   A very closely analogous statement is proved in \S 3 of \cite{She25} using a slightly different family of subgroups $U_{n, p}$, rather than the $K_{n, p}$ considered above. However, all that is needed for the arguments to work is that each of the subgroups in the tower contains $\overline{\cE}_{K^{(p)}}$, which is the input for the crucial Proposition 2.4 of \cite{She25}; this is true by definition for $K_{n, p}$, so the argument works for these subgroups without change. (Compare Proposition 2.7.9 of \cite{LRZ24}, which is a similar statement for locally symmetric spaces for any group satisfying an additional condition --Milne's axiom ``SV5'' (\cite[pp.75]{Milne}) -- which implies that the analogue of the group $\overline{\cE}_{K^{(p)}}$ for these groups is trivial.)
  \end{proof}
  
 \subsection{Moment maps}
 
  For a pure, dominant integral weight $\lambda$, we have a \emph{moment map}
  \[ \mom^{\lambda}_n : H^i_{\Iw}\left(Y_G(K_\infty)_{\Qbar}, \cO\right) \to H^i_{\et}\left(Y_G(K_n)_{\Qbar}, \sH^{[\lambda]}_{\cO}\right)\]
  defined by cup-product with the highest-weight vector of $V_\lambda^\vee$ (regarded as a section of $\sH^{[\lambda]}_{\cO}$ over $Y_G(K_\infty)$) composed with projection to level $K_n$. Since the highest weight of $V_\lambda^\vee$ is $-w_0 \lambda$, and we have conjugated the $\fS$-action by $w_0$, this map factors through the tensor product $H^i_{\Iw}\left(Y_G(K_\infty)_{\Qbar}, \cO\right) \otimes_{\Lambda_n} \cO[\lambda]$. Moreover, it is compatible with the actions of $\Gal(\Qbar/\QQ)$, of $\TT$, and of the $U_\fp'$ for $\fp \mid p$ (since our normalization of $U_\fp'$ makes it act trivially on the highest-weight vector).
  
  \begin{theorem}[Derived control theorem]
   \label{thm:classicity}
   For any $n \ge 0$ there is a spectral sequence (supported in $i \le 0, j \ge 0$)
   \[ E_2^{ij} = \operatorname{Tor}_{-i}^{\Lambda_n}\Big(
    e'_{\ord} H^i_{\Iw}(Y_G(K_\infty)_{\Qbar}, \cO), 
    \cO[\lambda]\Big) \Longrightarrow 
    e'_\ord H^{i + j}_{\et}\left(Y_G(K_n)_{\Qbar}, \sH^{[\lambda]}_{\cO}\right),\]
   in which the edge maps $E_2^{0j} \to E_\infty^j$ are the maps
   \[ \tag{\dag}
    \mom_n^{\lambda} : e'_{\ord}H^j_{\Iw}(Y_G(K_\infty)_{\Qbar}, \cO) \otimes_{\Lambda_n} \cO[\lambda] \to e'_{\ord} H^j_{\et}\left(Y_G(K_n)_{\Qbar}, \sH^{[\lambda]}_{\cO}\right). 
   \]
  \end{theorem}
  
  \begin{proof}
   This is proved in exactly the same way as Corollary 3.14 of \cite{She25}, which is the analogous statement for the level tower $(U_{n, p})_{n \ge 0}$ rather than $(K_{n, p})_{n \ge 0}$.
  \end{proof}
  
  \begin{remark}
   It follows from the Auslander--Buchsbaum formula that after inverting $p$, the set of $\lambda$ for which the $E_2^{ij}$ term is non-vanishing is contained in a closed subset of $\Spec \Lambda[1/p]$ of codimension at least $|i|$. So the map $(\dag)$ is bijective away from a codimension $\ge 1$ set, and injective away from a codimension $\ge 2$ set.
   
   Exactly as in \cite{She25}, using the cohomological vanishing results of Caraiani--Tamiozzo \cite{CT21}, if $\fm$ is a maximal ideal of $\TT \otimes \overline{\FF}_p$ whose associated mod $p$ Galois representation has non-solvable image, the localization $e'_{\ord} H^j_{\et}(Y_G(K_\infty)_{\Qbar}, \cO)_{\fm}$ is zero for $j \ne d$, and projective over $\Lambda$ for $j = d$. So the localization of $(\dag)$ at $\fm$ is an isomorphism for every $\lambda$. This recovers Dimitrov's exact control theorem \cite[Theorem 3.8(iii)]{Dim13}, under considerably weaker assumptions on $\fm$.
  \end{remark}

 \subsection{Existence of families}
 
  In this section we suppose our prime-to-$p$ level $K^{(p)}$ is given by $U_1(\fN)^{(p)}$ for some ideal $\fN \trianglelefteqslant \cO_F$ coprime to $p$, where
  \[ 
   U_1(\fN)^{(p)} = \left \{\gamma \in \GL_2(\widehat{\cO}_F^{(p)}): 
   \gamma \equiv \tbt{*}{*}{0}{1} \bmod \fN \right\}. 
  \]
  For simplicity, we suppose that $\fN$ does not divide 2, 3 or $\Delta_F$, which implies that $U_1(\fN)$, and hence any subgroup of it, is sufficiently small \cite[Lemma 2.1]{Dim09}. (The general case can be treated by taking invariants under a suitable auxiliary subgroup.)
  
  We now formulate the following definition. Let $\mathcal{C} = \operatorname{Max} A \subset \cW = (\operatorname{Spf} \Lambda)^\rig$ be an affinoid disc. To avoid vacuous cases we suppose that $\mathcal{C}$ contains at least one classical weight; it follows that classical weights are in fact Zariski-dense in $\mathcal{C}$ (at least if Leopoldt's conjecture holds for $F$).
  
  \begin{definition}
   An \emph{ordinary family of cuspidal automorphic representations} $\uPi$ over $\mathcal{C}$ consists of the following data:
   \begin{itemize}
    \item a ring homomorphism $\uPi : \TT \otimes \ZZ[\{ U_{\fp} : \fp \mid p\}] \to A$,
    \item a direct summand $M({\uPi})^*$ of $e'_{\ord} H^d_{\et}(Y_G(K_\infty)_{\Qbar}, \cO)(2) \otimes_{\Lambda} A$,
   \end{itemize}
   with the following properties:
   \begin{itemize}
    \item $M(\uPi)^*$ is free of rank $2^d$ over $A$ and stable under the actions of $\Gal(\overline{\QQ} / \QQ)$, $\TT$, and the $U_{\fp}'$ for $\fp \mid p$, with the Hecke operators acting via $\uPi \circ \nu$, where $\nu$ is the inverse map on the Hecke algebra.
   \item For each dominant integral weight $\lambda \in \mathcal{C}$, the specialization of $\uPi$ at $\lambda$ is the Hecke eigenvalue system associated to the ordinary refinement of $\Pi[\lambda]$, for some ordinary cuspidal automorphic representation $\Pi[\lambda]$ of weight $\lambda$ and conductor $\fN$ (with coefficient field contained in $E$).
   \item For each such $\lambda$, the map $\mom^\lambda_0$ restricts to an isomorphism between the weight $\lambda$ specialization of $M(\uPi)^*$, and the direct summand 
   \[ M(\Pi[\lambda], \alpha[\lambda])^* \subseteq e'_{\ord} H^d_{\et}(Y_G(K_0)_{\Qbar}, \sH^{[\lambda]})(2)\]
   defined above.
  \end{itemize}
  \end{definition}
  
  (Note that $M({\uPi})^*$ is not \emph{a priori} defined as the dual of anything, but it will interpolate the duals of Asai Galois representations, hence the notation.)
    
  Now let $\Pi$ be a cuspidal automorphic representation which is cohomological of weight $\lambda$, which has conductor $\fN$ and Hecke eigenvalues of $\Pi$ lie in $E$, and which is ordinary at $p$ (i.e.~$p^{-\sum_i t_i} \cT_p$ acts as a unit).
  
  \begin{theorem}\label{thm:family}
   There exists an open disc $\mathcal{C}$ and family $\uPi$ over $\mathcal{C}$ whose specialization in weight $\lambda$ is the given $\Pi$.
  \end{theorem}
  
  \begin{proof}
   Let $M^i_{\fP}$ be the localization of $e'_{\ord} H^j_{\Iw}(Y_G(K_\infty)_{\Qbar}, \cO)$ at the ideal $\mathfrak{M}_{\Pi^\vee} \otimes \fP$ of $\TT \otimes_{\ZZ} \Lambda$, where $\mathfrak{M}_{\Pi^\vee} = \nu(\mathfrak{M}_{\Pi}) \triangleleft \TT\otimes E$ is the kernel of the Hecke eigenvalue system of $\Pi$, and $\fP \triangleleft \Lambda$ is the ideal corresponding to the character $\lambda$.
   
   Since localization is exact, we obtain a localized spectral sequence of finitely-generated $\Lambda_{\fP}$-modules with $E^2$ terms \[ \operatorname{Tor}_{-i}^{\Lambda_{\fP}}( M^j_{\fP}, \Lambda_{\fP}/\fP) \Rightarrow e'_{\ord} H^{i+j}_{\et}(Y_G(K_0), \sH^{[\lambda]})_{(\mathfrak{M}_{\Pi^\vee})}.\]
   The abutment is the $\Pi^\vee$-isotypic part of the cohomology, which is zero outside degree $d$ and has dimension $2^d$ in degree $d$. A calculation using Nakayama's lemma and the local criterion for flatness shows that $M^i_{\fP}$ is likewise zero for $i \ne d$ and free of rank $2^d$ over $\Lambda_{\fP}$ for $i = d$.
   
   This gives a direct summand of $M^d$ with the required properties after localizing to some Zariski-open neighbourhood of $\lambda$; and we can choose our affinoid disc $\mathcal{C}$ to be contained in this, since the affinoid $G$-topology is finer than the Zariski topology.
  \end{proof}
%
%
  
  \begin{corollary}\label{cor:classicalinterp}
   For a family $\uPi$ as in \cref{thm:family}, at each classical point $\mu \in \mathcal{C}$, the map $(\Pr_{\fp, \alpha})_* \circ \mom^\mu_0$ is an isomorphism between the specialization of $M(\uPi)^*$ at $\mu$ and $M_E(\Pi[\mu])^*$.
  \end{corollary}
  
  \begin{proof}
   \cref{thm:family} identifies the weight $\mu$ specialization with $M_E(\Pi[\mu]; \alpha[\mu])^*$, where $\alpha[\mu]$ is the ordinary $p$-refinement of $\Pi[\mu]$. We have seen that $\Pr_{p, \alpha, *}$ identifies this space with $M_E(\Pi[\mu])^*$.
  \end{proof}
 
  \begin{remark}
   One can formulate more general results interpolating representations $\Pi$ which are ramified at $p$, but have non-zero $K_{n, p}$-invariants for some $n$ with ordinary $U_p$-eigenvalues; these are precisely the representations which are \emph{nearly ordinary} in the sense of \cite{Dim13}. We leave the details to the interested reader.
  \end{remark}

\section{Asai--Flach classes}

 We now assume (for the remainder of this paper) that $[F : \QQ] = 2$. 
 
 \subsection{Spherical and Iwahori-level classes} 
 
  We now briefly recall the cohomology classes constructed in \cite{LLZ18} that we shall interpolate $p$-adically. Let $\Pi$ be an automorphic representation of conductor $\fN$ (coprime to $p$), cohomological in weight $\lambda = (k_1, k_2, t_1, t_2)$.
 
  \begin{definition}
   For $0 \le j \le \min(k_1, k_2)$, let us write $h = t_1 + t_2 + j$, and let
   \[ \AF^{[\Pi, j]}_{\et} \in H^1(\QQ, M_E(\Pi)^*(-h)) \]
   be the class defined in \cite[Definition 4.4.6]{LLZ18} for the newform $\cF \in \Pi^{U_1(\fN)}$ generating $\Pi$.
  \end{definition}
  
  This is well-defined, since the space $M_E(\Pi)^*$ is identical to the $M_E(\cF)^*$ of \emph{op.cit.} with its Galois action twisted by $t_1 + t_2$. We also need to consider ``$p$-stabilized'' versions of these classes. Let $\alpha = (\alpha_{\fp})_{\fp \mid p}$ be a $p$-refinement of $\Pi$.
   
  \begin{definition}
   For $j$ as before, let
   \[ \AF^{[\Pi, \alpha, j]}_{\et} \in H^1(\QQ, M_E(\Pi, \alpha)^*(-h)) \]
   be the class defined \emph{loc.cit.}~taking $\cF$ to be the ordinary $p$-stabilized eigenform $\cF_\alpha \in \Pi^{K_0}$ corresponding to $\alpha$.
  \end{definition}
  
  It is important to note that although $\Pr_{p, \alpha, *}$ maps $M_E(\Pi, \alpha)^*$ isomorphically to $M_E(\Pi)^*$, the two versions of the Asai--Flach class do not match up; instead, they satisfy the following more subtle compatibility.
  
  \begin{theorem}[$p$-stabilization relation]\label{thm:pstab}
   These classes are related by
   \[ \left(\operatorname{Pr}_{p, \alpha}\right)_* \left( \AF^{[\Pi, \alpha, j]}_{\et}\right) = \cR(p^{-(1+h)}) \cdot \AF_{\et}^{[\Pi, j]}, \]
   where $\cR(X)$ is the following polynomial:
   \begin{itemize} 
   \item for $p$ inert,
   \[ \cR(X) = \left( 1 - \beta_p X\right) \left( 1 - \alpha_p \beta_p X^2\right) \]
   \item for $p = \fp_1 \fp_2$ split,
   \[ \cR(X) = \left( 1 - \alpha_{\fp_1}\beta_{\fp_2} X\right) \left( 1 - \beta_{\fp_1} \alpha_{\fp_2} X\right)\left( 1 - \beta_{\fp_1} \beta_{\fp_2} X\right). \]
   \end{itemize}
  \end{theorem}
  
  Since this requires rather different methods than the $p$-adic interpolation calculations making up the majority of this paper, we shall defer its proof to \cref{sect:pstab} below.
  
\section{Families of Asai--Flach elements}

 \subsection{The groups $G^*$ and $G$}
  
  \begin{definition} \label{def:Gstar}
   Let $G^*$ be the group $G \times_{(\Res^F_\QQ \GG_m)} \GG_m$, and $Y_1^*(\fN p^n)$ the Shimura variety for $G^*$ of level $U_1^*(\fN p^n) \coloneqq U_1(\fN p^n) \cap G^*$.
  \end{definition}
  
  \begin{proposition}
   For each $n \ge 1$ there is a canonical map of $\QQ(\zeta_{p^n})$-varieties
   \[ 
    \jmath_n : Y_1^*(\fN p^n) \times_{\QQ} \QQ(\zeta_{p^n}) \to Y_G(K_n) \times_{\QQ} \QQ(\zeta_{p^n}). 
   \]
  \end{proposition}
  
  \begin{proof}
   We consider the auxiliary group $\widetilde{G} \coloneqq G \times \GG_m$, and the map $\tilde\jmath = (\jmath, \det) : G^* \hookrightarrow \widetilde{G}$, where $\jmath$ is the inclusion $G^* \hookrightarrow G$.
   
   As in \cite{LLZ18}, we can identify $Y^*_1(\fN p^n) \times_{\QQ} \QQ(\zeta_{p^n})$ with the Shimura variety for $G^*$ of level
   \[ U^*_{11}(\fN p^n) \coloneqq \{ g \in U_1(\fN p^n) : \det g = 1 \bmod p^n\}. \]
   Likewise, we may clearly identify $Y_G(K_n) \times_{\QQ} \QQ(\zeta_{p^n})$ with the Shimura variety for $\widetilde{G}$ of level $K_n \times V_n$, where $V_n = (1 + p^n \Zp)^\times$. Since we have $\tilde{\jmath}\left(U^*_{11}(\fN p^n)\right) \subseteq K_n \times V_n$, the required homomorphism follows from the functoriality of canonical models of Shimura varieties.
  \end{proof}
  
  Analogously, for level $n = 0$, we write $U_1^*(\fN (p))$ for $U_1^*(\fN) \cap \Iw_p$, and we have maps (of $\QQ$-varieties)
  \[ \jmath_0 : Y_1^*(\fN(p)) \to Y_G(K_0), \qquad \jmath_{\sph} : Y_1^*(\fN) \to Y_G(K_{\sph}). \]
  
 \subsection{Iwasawa-theoretic classes}
   
  Theorem 8.2.3 of \cite{LLZ18} gives a canonical ``Asai--Flach'' class
  \[ {}_c \mathcal{AF}_{1, \fN p, a} \in 
   H^1_{\Iw}\Big(\QQ(\mu_{p^\infty}), 
    e'_\ord H^2_{\Iw}\big(Y_1^*(\fN p^\infty)_{\Qbar}, \cO\big)(2)\Big).\]
  
  \begin{remark}
   The class ${}_c \mathcal{AF}_{1, \fN p, a}$ is defined in \emph{op.cit.}~in the cohomology of an integral model of $Y_1^*(\fN p)$ over $\ZZ[1/\Sigma]$, with values in a sheaf of Iwasawa modules $\Lambda_R(\sH_{\cA})(2)$. However, in fact it takes values in a direct summand $\Lambda_R(\sH_{\cA}\langle t_p\rangle)$ of this sheaf, by Remark 6.4.6 of \emph{op.cit.}; this sheaf is the inverse limit as $n \to \infty$ of the pushforwards of the constant sheaf $\underline{\cO}$ from $Y_1^*(\fN p^n)$ to $Y_1^*(\fN p)$, so the cohomology of this sheaf coincides with the inverse-limit cohomology used here. (Compare e.g.~\cite[\S 4]{KLZ17}.)
   
   The class of \emph{op.cit.}~lies in $H^3$ of this integral model, which is related to $\Gal(\QQ_{\Sigma}/\QQ)$-cohomology of the variety over $\Qbar$ via the Hochschild--Serre spectral sequence. However, since $H^0_{\Iw}(\QQ(\mu_{p^\infty}), -)$ vanishes for every Galois representation, this class is homologically trivial (its image under base-extension to $\Qbar$ is zero); so we may map it into $H^1_{\Iw}$ as above. 
   
   Finally, $a$ denotes a choice of basis of $\cO_F / \ZZ$; but there is a canonical choice, characterized by $\sigma_1(a) - \sigma_2(a) = \sqrt{\Delta_F}$, and we shall fix this choice and drop it from the notation.
  \end{remark}
  
  \begin{definition}
   Let
   \[ z_\infty \in H^1_{\Iw}\Big(\QQ(\mu_{p^\infty}), e'_\ord H^2_{\Iw}\big(Y_G(K_\infty)_{\Qbar}, \cO\big)(2)\Big) \]
   denote the image of ${}_c \mathcal{AF}_{1, \fN p}$ under $\jmath_{\infty, *} = (\jmath_{n,*})_{n \ge 1}$.
  \end{definition}

  \begin{remark}
   Since each $\jmath_n$ is only defined over $\QQ(\mu_{p^n})$, and does not descend to $\QQ$, the map
   \[ 
    \jmath_{\infty, *} : e'_\ord H^2_{\Iw}\big(Y_1^*(\fN p^\infty)_{\Qbar}, \cO\big)
    \to e'_\ord H^2_{\Iw}\big(Y_G(K_\infty)_{\Qbar}, \cO\big)
   \]
   is only Galois-equivariant up to a twist by a $\Lambda^\times$-valued character of $\Gal(\QQ(\mu_{p^\infty}) / \QQ)$, twisting the weight $\lambda$ specialization by $t_1 + t_2$.
  \end{remark}
  
 \subsection{Compatibility of moment maps}
 
  \begin{proposition} \label{mmcompatibility}
   If $\lambda^* = \lambda|_{G^*}$, then the following diagram commutes, for any $n \ge 1$:
   \[ 
    \begin{tikzcd}
     H^2_{\Iw}(Y_1^*(\fN p^\infty)_{\Qbar}, \cO) \arrow{r}{\jmath_{\infty, *}} \arrow[swap]{d}{\mom^{\lambda^*}_n} & 
     H^2_{\Iw}(Y_G(K_\infty)_{\Qbar}, \cO) \arrow{d}{\mom^{\lambda}_n} \\
     H^2_{\et}(Y_1^*(\fN p^n)_{\Qbar}, \sH^{[\lambda^*]}) \arrow{r}{\jmath_{n, *}}& H^2_{\et}(Y_G(K_n)_{\Qbar}, \sH^{[\lambda]}) 
    \end{tikzcd}
   \]
  \end{proposition}
 
  \begin{proof}
   For each $m \ge n$, let $b^\lambda_m \in H^0(Y_G(K_m), \sH^{[\lambda]}_{\cO/p^m})$ denote the mod $p^m$ reduction of the highest-weight vector, and similarly $b^{\lambda^*}$ for $G^*$. Then we have the compatibility
   \(
     b^{\lambda^*}_m = \jmath_m^*(b^\lambda_m),
   \)
   so by the push-pull formula in \'etale cohomology, for any $y_m \in H^2_{\et}(Y_1^*(\fN p^m)_{\Qbar}, \cO/p^m)$ we have
   \begin{align*}
    \jmath_{m,*}(y_m \cup b^{\lambda^*}_m) &= \jmath_{m, *}(y_m \cup \jmath_m^*(b^{\lambda}_m)) \\
    &= \jmath_{m,*}(y_m) \cup b^{\lambda}_m.
   \end{align*}
   Since the moment map $\mom_n^{\lambda^*}$ sends $(y_m)_{m \ge 1}$ to the limit of the sequence 
   \[ 
    \left(\operatorname{norm}^{\fN p^m}_{\fN p^n}(y_m \cup b^{\lambda^*}_m)\right)_{m \ge n} \in \varprojlim_{m \ge n} H^2_{\et}\left(Y_1^*(\fN p^n)_{\Qbar}, \sH^{[\lambda]}_{\cO/p^m}\right) =  H^2_{\et}\left(Y_1^*(\fN p^n)_{\Qbar}, \sH^{[\lambda]}_{\cO}\right),
   \]
   and similarly for $G$, the result follows.
  \end{proof}
 
  \begin{theorem}[Theorem B] \label{b2}
   \label{prop:interp1}
   Let $\Pi$ be a cuspidal automorphic representation of $G$ of weight $\lambda$ and level $\fN$, unramified and ordinary at $p$. Then the map of Galois representations
   \begin{align*}
    e'_\ord H^2_{\et}\big(Y_G(K_\infty)_{\Qbar}, \cO\big)(2) \otimes_{\Lambda} E[\lambda]
    \xrightarrow{\mom^\lambda_0}&\ 
    e'_\ord H^2_{\et}\big(Y_G(K_0)_{\Qbar}, \sH^{[\lambda]}\big)(2)\\
    \twoheadrightarrow&\ M_E(\Pi, \alpha)^*
   \end{align*}
   sends $z_\infty$ to $\operatorname{tw}_{(t_1 + t_2)}\left({}_c\AF^{\cF_{\alpha}}_1\right)$, where: $\cF_{\alpha}$ is the ordinary $p$-refinement of the newform generating $\Pi$; ${}_c\AF^{\cF_{\alpha}}_{1}$ is the Iwasawa cohomology class associated to this $p$-stabilized eigenform as in Theorem 9.1.2 of \cite{LLZ18}; and $\operatorname{tw}_{(t_1+t_2)}$ denotes the canonical isomorphism between the Iwasawa cohomology of $M_E(\cF_\alpha)^*$ and its Tate twist $M_E(\cF_\alpha)^*(t_1 + t_2) = M_E(\Pi, \alpha)^*$. 
  \end{theorem}
  
  \begin{proof}
   This follows from the compatibility of moment maps in the previous proposition and the definition of the classes ${}_c\AF^{\cF_{\alpha}}_1$, which are constructed as the images of ${}_c \mathcal{AF}_{1, \fN p}$ under the $G^*$ moment maps.
  \end{proof}
%
 
 \subsection{Asai--Flach classes for families}
  
  Let us now fix an ordinary family $\underline{\Pi}$ over some open $\mathcal{C}$, as above.
  
  \begin{definition}
   We let  
   \[ {}_c \mathcal{AF}(\uPi) \in H^1_{\Iw}(\QQ(\mu_{p^\infty}), M(\uPi)^*) \]
   be the image of $z_\infty$.
  \end{definition}
  
  By the preceding proposition, the specialization of ${}_c \mathcal{AF}(\uPi)$ at any $\lambda \in 
  \mathcal{C}$ is the class $\operatorname{tw}_{(t_1 + t_2)}\left({}_c\AF^{\cF_{\alpha}}_1\right)$, for $\cF_\alpha$ the ordinary eigenform generating $\Pi[\lambda]$; and likewise for ${}_c \mathcal{AF}^{\ad}(\uPi)$.
  
 We now show that these interpolate the ``finite-level'' Asai--Flach classes associated to suitable twists of Hilbert eigenforms (as defined in \cite[Definition 4.4.6]{LLZ18}). We can interpret the Iwasawa cohomology as a module over $\Lambda \mathop{\hat\otimes} \Lambda^{\mathrm{cy}}$, where $\Lambda^{\mathrm{cy}} = \cO[[\Zp^\times]]$ is the cyclotomic Iwasawa algebra, whose rigid-analytic spectrum $\cW^{\mathrm{cy}}$ parametrizes characters of $\Zp^\times$.
  
  \begin{theorem} \label{int1}
   Let $\lambda = (k_1, k_2; t_1, t_2)$, and suppose $0 \le j \le \min(k_1, k_2)$. As before write and $h = t_1 + t_2 + j$. Then evaluation at $(\lambda, h) \in \mathcal{C}\times \cW^{\mathrm{cy}}$ maps ${}_c \mathcal{AF}(\uPi)$ to the class
   \[
    \frac{(c^2 - c^{2(h-w)}\omega_{\Pi}(c))\left(1 - \frac{p^{h}}{\alpha_p} \right)}{\sqrt{\Delta_F}^j (-1)^j j! \binom{k}{j} \binom{k'}{j}}  \AF^{[\Pi, \alpha, j]}_{\et} \in H^1(\QQ, M_E(\Pi, \alpha)^*(-h)). 
   \]
  \end{theorem}
  
  \begin{proof}
   This follows by combining \cref{prop:interp1} with the interpolation property of the classes ${}_c \mathcal{AF}$ given in \cite[Theorem 9.1.2]{LLZ18}.
  \end{proof}
  
  Identifying $M_E(\Pi, \alpha)^*$ with $M_E(\Pi)^*$, and using \cref{thm:pstab}, we can write this as follows:
  
  \begin{theorem}[Theorem C] \label{c2}
   With the above notations we have
   \[ \operatorname{sp}_{(\lambda, h)} \left({}_c \mathcal{AF}(\uPi)\right) = 
   \frac{(c^2 - c^{2(h-w)}\omega_{\Pi}(c))(1 - \tfrac{p^h}{\alpha_p}) \cR(p^{-1-h})}{\sqrt{\Delta_F}^j (-1)^j j! \binom{k}{j} \binom{k'}{j}}  \AF^{[\Pi, j]}_{\et},\]
   where $\cR$ is the Euler factor from \cref{thm:pstab}.\qed
  \end{theorem}
  
  \begin{remark}
   Note that there is some ``redundancy'' here, since twisting $\Pi$ by powers of the norm character corresponds to a cyclotomic twist of the Galois representation, while Iwasawa cohomology over $\QQ(\mu_{p^\infty})$ also interpolates cyclotomic twists. Concretely, we get the same cohomology class if we specialize at $(\lambda, h)$ or at $(\lambda + (0, 0, 1, 1), h + 2)$. This redundancy can be eliminated by restricting $\lambda$ to lie in the codimension 1 subspace of $\cW$ parametrizing characters with a fixed value of $w$ (so we are varying $\Pi$ in a family while keeping the $\infty$-type of its central character fixed). The case $w = 0$ corresponds to the theory of \cite{She25}.
   
  \end{remark}
\section{Proof of the $p$-stabilization relation}
 \label{sect:pstab}

 \subsection{Inert primes} 
 
  For $p$ inert in $F$, we can give a direct proof of \cref{thm:pstab} using the formulae for Hecke actions on Asai--Flach classes from \cite{LLZ18}.

  \begin{proof}[Proof of Theorem \ref{thm:pstab} for $p$ inert]
   Recall that $\Pr_{p, \alpha} = \Pr_{p, 1} - \tfrac{1}{\alpha} \Pr_{p, 2}$; and $\AF^{[\Pi, \alpha, j]}$ is the projection to the $\cF_\alpha$-eigenspace of the class $\AF^{[k_1, k_2, j]}_{\et, \fN p}$ in the notation of \emph{op.cit.}, and likewise for $\AF^{[\Pi, j]}$.
   
   Theorem 7.1.2(a) of \emph{op.cit.} gives us the formula
   \[ \Pr_{p, 1, *}\left( \AF^{[k_1, k_2, j]}_{\et, \fN p} \right) = 
   \left(1 - p^{-2h} \cS_p^{-1} \right) \AF^{[k_1, k_2, j]}_{\et, \fN}. \]
   For $\Pr_{p, 1, *}$, we can apply Corollary 7.4.2 of \emph{op.cit.}, but with a slight correction since our conventions for $\Pr_{p, 2, *}$ are not quite the same as the $(\pr_{2, p})_*$ of \emph{op.cit.}; with our conventions the formula becomes 
   \[ \Pr_{p, 2, *}\left( \AF^{[k_1, k_2, j]}_{\et, \fN p} \right) =
    p^{1-h} \cS_p^{-1} \left(1 - p^{-2h} \cS_p^{-1} \right) \AF^{[k_1, k_2, j]}_{\et, \fN}. 
   \]
   Combining these, and noting that $\cS_p^{-1}$ acts as $p^{-2}\alpha\beta$ on $M(\Pi)^*$, we see that the projection of $\Pr_{p, \alpha,*}\left(\AF^{[k_1, k_2, j]}_{\et, \fN p}\right)$ to $M(\Pi)^*$ is given by the expected formula
   \[ \left(1 - p^{-1-h}\beta\right) \left(1 - p^{-2h-2} \alpha \beta \right) \AF^{[\Pi, j]}_{\et}.\qedhere\]
  \end{proof}

 \subsection{Proof for $p$ split}

  In the case of split primes, the theorem is substantially more difficult. If the primes above $p$ are trivial in the narrow class group (so the degeneracy maps involved in the theorem preserve the image of $Y_{G^*}$ in $Y_G$), then one can give a proof similar to the above. However, for primes $p$ not satisfying this assumption different methods are required, so we shall instead use the representation-theoretic techniques based on multiplicity-one theorems developed in \cite{LSZ21}.
  
  \begin{notation} \
   \begin{itemize}
   \item Let $\cW(\Pi_p)$ denote the Whittaker model of $\Pi_p$, with respect to some choice of unramified additive character of $F \otimes \Qp$ trivial on $\Qp$.
   \item Let $W^{\sph} \in \cW(\Pi_p)$ be the normalized spherical vector, and 
   $W_\alpha = \Pr_{p, \alpha} \cdot W^{\sph}$ the normalized $U_p$-eigenvector in $\cW(\Pi_p)^{\Iw_p}$.
   \item Let $\cS(\Qp^2)$ be the Schwartz space of locally constant, compactly supported functions on $\Qp^2$.\qedhere
   \end{itemize}
  \end{notation}

  \begin{proposition}
   There exists a map
   \[ \mathrm{Z}^{\mathrm{mot}, j} : \cW(\Pi_p) \times \cS(\Qp^2) \to H^1(\QQ, M_E(\Pi)^*(-h)) \]
   which transforms under $g \in \GL_2(\Qp)$ by
   \[ \mathrm{Z}^{\mathrm{mot}, j}(gW, g\Phi) = |\det g|^{-h} \mathrm{Z}^{\mathrm{mot}, j}(W, \Phi).\]
   Morever, we have
   \[ \mathrm{Z}^{\mathrm{mot}, j}(W^{\sph}, \ch(\Zp^2)) = \AF^{[\Pi, j]}_{\et}, \]
   whereas
   \[ \mathrm{Z}^{\mathrm{mot}, j}\Big(W_\alpha, \ch((0, 1) + p\Zp^2)\Big) = \tfrac{1}{(p^2 - 1)} \operatorname{Pr}_{p, \alpha, *} \left( \AF_{\et}^{[\Pi, \alpha, j]}\right). \]
  \end{proposition}

  \begin{proof}
   The construction of $\mathrm{Z}^{\mathrm{mot}, j}$ amounts to nothing more than carefully keeping track of all the choices made during the construction of the Euler system class; compare \cite{LZ24} for the $\operatorname{GSp}_4$ Euler system. See \cite{Gro20} for a similar result phrased in terms of Hecke algebras for $G$ (one can translate between that formulation and the one above using Frobenius reciprocity).
  \end{proof}
  
  \begin{lemma}
   The space of linear functionals on $\cW(\Pi_p) \times \cS(\Qp^2)$ with the $\GL_2(\Qp)$-equivariance property described above is 1-dimensional, and spanned by the Rankin--Selberg zeta integral
   \[ Z^{\mathrm{an}, j}(W, \Phi) \coloneqq
   \lim_{s \to h} \frac{1}{L^{\As}(\Pi_p, s)} \int_{N(\Qp) \backslash \GL_2(\Qp)}W(g) f^{\Phi}(g; s)\, \mathrm{d}g
   \]
   where $f^{\Phi}$ is a Godement--Siegel section (and $N$ the upper-triangular unipotent).
  \end{lemma}
  
  \begin{proof}
   Any such linear functional must factor through the maximal quotient $\sigma$ of $\cS(\Qp^2)$ on which the centre of $\GL_2(\Qp)$ acts via the character $|\cdot|^{-2h} \omega_{\Pi}^{-1}$. This quotient is either an irreducible principal series, or a reducible principal series with 1-dimensional quotient; in either case it is a Whittaker-type representation, so the main result of \cite{harrisscholl} shows that $\Hom_{\GL_2(\Qp)}(\Pi_p \times \sigma |\cdot|^h, \CC)$ is at most 1-dimensional. To complete the proof it suffices to check that $Z^{\mathrm{an}, j}(W, \Phi)$ has the required equivariance property (which is obvious) and is non-zero, which follows from the fact that it maps the spherical data to 1.
  \end{proof}
  
  So to prove Theorem \ref{thm:pstab}, which is a proportionality relation between values of the motivic linear functional $\mathrm{Z}^{\mathrm{mot}, j}$, it suffices to show that the corresponding linear relation holds between values of the analytic linear functional $\mathrm{Z}^{\mathrm{an}, j}$ instead. It is well known that we have $\mathrm{Z}^{\mathrm{an}, j}(W^{\sph}, \ch(\Zp^2)) = 1$, while an elementary computation gives
  \[ Z^{\mathrm{an}, j}(W_\alpha, \ch((0, 1) + p\Zp^2)) = 
   \frac{1}{(p^2 - 1)} \lim_{s \to h} \frac{\left(1 - \alpha_{\fp_1} \alpha_{\fp_2} p^{-1-s}\right)^{-1}}{L^{\As}(\Pi, s)} = \frac{1}{(p^2 - 1)}\mathcal{R}(p^{-1-h}). \]
  This gives the result claimed.

\end{document}